\documentclass[12pt]{amsart}
\usepackage{amssymb}
\usepackage{enumerate,color}
\usepackage{graphicx}
\usepackage[text={6in,9in},centering]{geometry}
\usepackage{subfigure}
\usepackage{epstopdf}

\def\bbbr{{\rm I\!R}}

\theoremstyle{plain}
  \newtheorem{thm}{Theorem}[section]
  \newtheorem{cor}[thm]{Corollary}
  \newtheorem{lem}[thm]{Lemma}
  
\theoremstyle{definition}

\theoremstyle{remark}

\DeclareMathOperator\bdim{dim_B}

\def\R{\mathbb{R}}
\def\D{\mathcal{D}}

\def\card{\textrm{card}}

\newcommand{\bw}{\mathbf{w}}
\newcommand{\bp}{\mathbf{p}}
\newcommand{\dimh}{\dim_{\rm H}}

\newcommand\asd{\mbox{\rm dim}_{\rm A}\,} % Assouad dimension
\newcommand\lwd{\mbox{\rm dim}_{\rm L}\,} % lower dimension
\newcommand\pkd{\mbox{\rm dim}_{\rm P}\,} % packing dimension
\newcommand\hdd{\mbox{\rm dim}_{\rm H}\,} % Hausdorff dimension
\newcommand\ubd{\overline{\mbox{\rm dim}}_{\rm B}\,} % upper box dimension
\newcommand\lbd{\underline{\mbox{\rm dim}}_{\rm B}\,} % lower box dimension
\newcommand\uld{\overline{\mbox{\rm dim}}_{\rm loc}\,} % upper local dimension
\newcommand\lld{\underline{\mbox{\rm dim}}_{\rm loc}\,} % lower local dimension
\newcommand\lcd{\mbox{\rm dim}_{\rm loc}\,} %  local dimension

\begin{document}
\title[Dimensions  of a class of  self-affine Moran sets and measures in $\R^2$ ]{Dimensions  of a class of  self-affine Moran sets and measures in $\R^2$ }
\author{Yifei Gu}
\address{Department of Mathematics, East China Normal University, No. 500, Dongchuan Road, Shanghai 200241, P. R. China}

\email{52275500012@stu.ecnu.edu.cn}

\author{Chuanyan Hou}
\address{College of Mathematics Sciences, Xinjiang Normal University, Urumqi, Xinjiang, 830054, P. R. China}
\email{hchy\_e@163.com}

\author{Jun Jie Miao}
\address{Department of Mathematics, East China Normal University, No. 500, Dongchuan Road, Shanghai 200241, P. R. China}

\email{jjmiao@math.ecnu.edu.cn}
%\thanks{The research of Miao is partially supported by Shanghai Key Laboratory of PMMP (18dz2271000).}
%\subjclass[2000]{28A80}

%\keywords{}

\begin{abstract}
For each integer $k>0$, let $n_k$ and $m_k$ be integers such that $n_k\geq 2, m_k\geq 2$, and let $\mathcal{D}_k$ be a subset of $\{0,\dots,n_k-1\}\times
\{0,\dots,m_k-1\}$. For each $w=(i,j)\in \mathcal{D}_k$, we define an affine transformation on~$\R^2$ by
$$
  \Phi_w(x)=T_k(x+w), \qquad w\in\mathcal{D}_k,
$$
where $T_k=\operatorname{diag}(n_k^{-1},m_k^{-1})$.
The non-empty compact set
$$
E=\bigcap\nolimits_{k=1}^{\infty}\bigcup\nolimits_{(w_1w_2\ldots w_k)\in \prod_{i=1}^k\mathcal{D}_i} \Phi_{w_1}\circ \Phi_{w_2}\circ \ldots\circ \Phi_{w_k}
$$
is called a \textit{self-affine Moran set}.

In the paper, we provide the lower, packing, box-counting and Assouad  dimensions of the self-affine Moran set $E$. We also explore the dimension properties of self-affine Moran measure $\mu$ supported on $E$,  and we provide  Hausdorff, packing and entropy dimension formulas of $\mu$.
\end{abstract}

\maketitle

\section{Introduction}
\subsection{Dimensions of measures}
In the dimension theory of fractal geometry and dynamical systems, the dimensions of invariant measures are important objects to investigate, and the most frequently used dimensions are Hausdorff dimension  and entropy dimension.

Let $\mu$ be a finite Borel measure in $\R^d$. The \textit{Hausdorff and packing dimensions} of $\mu$, respectively, are defined as
$$
\hdd \mu=\inf\{\hdd A :  \mu(A^c)=0\}, \qquad \pkd \mu =\inf\{\pkd A :  \mu(A^c)=0\}.
$$
The \textit{lower and upper local dimensions} of $\mu$ are given by
$$
\lld \mu(x) = \liminf_{r\to 0} \frac{\log \mu(B(x,r))}{\log r},  \qquad
\uld  \mu(x) =\limsup_{r\to 0} \frac{\log \mu(B(x,r))}{\log r},
$$
and we say \textit{local dimension} exists at $x$ if these are equal, writing $\lcd \mu(x)$ for the common value.
Let $\mathcal{M}_n$ be the partition of $\R^d$ into grid boxes $\Pi_{i=1}^d [2^{-n}j_i, 2^{-n}(j_i+1)]$ with integers $j_i$.  The \textit{lower and upper entropy dimensions} of $\mu$, respectively, are defined as
$$
\underline{\dim}_e \ \mu = \liminf_{n \to \infty} \frac{H_n(\mu)}{ \log 2^n},  \qquad
\overline{\dim}_e \ \mu =\limsup_{n \to \infty} \frac{H_n(\mu)}{ \log 2^n},
$$
where
$$
H_n(\mu)=-\sum_{Q\in \mathcal{M}_n} \mu(Q)\log\mu(Q).
$$
If these are equal, we refer to the common value
as the \textit{entropy dimension} of $\mu$. We refer the reader to~\cite{Bk_KJF2,FanLR02} for the background reading.
%$$\dim_e  \mu = \lim_{n \to \infty} \frac{H_n(\mu)}{ \log 2^n}.$$

  A well known theorem of Young~\cite{Young82} States that
\begin{thm}
Let $\mu$ be a probability measure on $\R^d$. Suppose that the local dimension
$$
\lcd \mu(x)= \alpha. \qquad \mu\textit{-a.e. } x\in\R^d.
$$
Then $\dim_{e}\mu=\hdd \mu=\alpha$.
\end{thm}
Determination of the dimensions of fractal sets is a challenging problem, see~\cite{BHR19, DasSim17, Falco88,FDJWY05, FMS18, Solom98,PerSo00} for various studies on the dimension theory of fractal sets. In particular, for non-typical self-affine sets, one strategy is to compute the Hausdorff dimensions of measures supported on a fractal set via local dimensions, and the supreme dimension of measures often gives the Hausdorff dimension of the fractal set, see~\cite{BHR19, Baran07,Bedfo84, DasSim17, KenPer96, McMul84,LalGa92},
Hence, in these studies, people rely on the existence of local dimensions to compute the dimensions of measures, that is to say, in these studies the Hausdorff and entropy dimension of measure are identical. It is an interesting question to investigate the dimension theory of the sets and measures where local dimensions do not exist.

\subsection{Self-affine sets}  First, we review a class of non-typical self-affine sets.

Given  integers $m$ and $n$ such that $n\geq m\geq 2$.
Let $\mathcal{D}$ be a subset of $\{0,\dots,n-1\}\times
\{0,\dots,m-1\}$.
For each $w\in \mathcal{D}$, we define an
affine transformation $\Phi_{w}$ on~$\R^2$ by
\begin{equation}\label{eq:Sk}
  \Phi_w(x)=T(x+w),
\end{equation}
where $T=\operatorname{diag}(n^{-1},m^{-1})$. Then
$\{\Phi_w\}_{w\in\mathcal{D}}$ forms a self-affine \emph{iterated function system} (IFS).  By the well-known theorem of Hutchinson, see \cite{Bk_KJF2,Hutch81}, this self-affine IFS has a unique self-affine attractor,  that is a unique non-empty compact set  $E \subset
\bbbr^{N}$ such that
$$
E = \bigcup_{i=1}^{m} \Phi_{i}(E).
$$
The self-affine set $E$ is also called
a \emph{Bedford-McMullen set} or  a \emph{Bedford-McMullen carpet}  \cite{Bedfo84,McMul84}. Since Bedford-McMullen carpets are a class of simplest self-affine sets, They are frequently used as a testing ground on questions and conjectures of fractals.

All kinds of dimensions of Bedford-McMullen carpets have been investigated, see~\cite{Bedfo84, Fraser14, Fraser21, Mackay11, McMul84}, and these sets are often used as good examples for the following dimension inequalities
\begin{equation}\label{dim_ineq}
\lwd E\leq  \hdd E\leq \bdim E\leq \asd E,
\end{equation}
where $\lwd$ and $\asd$ denote lower dimension and Assouad dimension, respectively, see section~\ref{sec_Ld} for the definitions. We refer readers to~\cite{Fraser20} for details of Assouad type dimensions.

There are various  generalisations for Bedford-McMullen carpets from different aspects, see~\cite{Baran07, Fraser14, Fraser21, FDJWY05, KenPer96, LalGa92}. In~\cite{KenPer96}, Kenyon and Peres studied the self-affine sponge $E$, which is a generalization of Bedford-McMullen carpet in $\R^d$, and they find the Hausdorff dimension of self-affine measures by using ergodic property to show that the local dimension exists. Moreover, they proved that there exists a unique ergodic self-affine measure of full Hausdorff dimension, i.e.
$$ \hdd \mu=\max\{\hdd \nu : \textit{ for all self-affine measure $\nu$ supported on $E$}  \}=\hdd E.$$

In this paper, we study a class of new fractals, named self-affine Moran sets (see subsection~\ref{sec_SAMS}), which may also be regarded as a generalisation of Bedford-McMullen carpet. Since we apply different affine IFSs at the different levels in the iterating process,  such sets do not have dynamical properties any more. Therefore the tools of ergodic theory cannot be invoked, which causes that  the local dimension of measures supported on these sets does not exist, and this leads to the difficulties to determine their dimensions of the sets and measures.% We mainly investigate the dimensions of these sets and  measures supported on the sets.

\subsection{Self-affine Moran sets }\label{sec_SAMS}
Given  a sequence $\{(n_k,m_k)\}_{k=1}^\infty$, where $m_k$ and $n_k$ are integers such that $n_k\geq 2$ and  $m_k\geq 2$. For each integer $k>0$, let $\mathcal{D}_k$ be a subset of $\{0,\dots,n_k-1\}\times\{0,\dots,m_k-1\}$. We write $r_k=\card(\D_k)$ and always assume that $r_k\geq 2$. % It is clear that  $1< r_k\leq m_k n_k$.
The set of all  finite sequences with length $k$ and the set of infinite sequences are denoted by
$$
  \Sigma^{k} = \prod_{j=1}^k\mathcal{D}_j  ,\qquad
  \Sigma^{\infty} = \prod_{j=1}^\infty\mathcal{D}_j.
$$
For
$\mathbf{w}=w_1\cdots w_k\in\Sigma^k$,
$\tau=\tau_1\cdots\tau_l\in\Sigma^l$, write $\mathbf{w}\ast\tau=
w_1\cdots w_k\tau_1\cdots\tau_l\in\Sigma^{k+l}$. We write
$\mathbf{w}|k = (w_1\cdots w_k)$ for the {\it curtailment}
after $k$ terms of $\mathbf{w} = (w_1 w_2\cdots)\in
\Sigma^{\infty}$. We write $\mathbf{w} \preceq \tau$ if $\mathbf{w}$ is a
curtailment of $\tau$. We call the set $[\mathbf{w}] =
\{\tau\in\Sigma^{\infty} : \mathbf{w} \preceq\tau\}$ the {\it cylinder}
of $\mathbf{w}$, where $\mathbf{w}\in \Sigma^*$. If $\mathbf{w}=\emptyset$, its cylinder is $[\mathbf{w}]=\Sigma^{\infty}$.

Given an integer $k>0$. For each $w=(i,j)\in \mathcal{D}_k$, we define an affine transformation on~$\R^2$ by
\begin{equation}\label{eq:Sk}
  \Phi_w(x)=T_k(x+w), \qquad w\in\mathcal{D}_k,
\end{equation}
where $T_k=\operatorname{diag}(n_k^{-1},m_k^{-1})$.
For each $\mathbf{w}=(w_1w_2\ldots w_k)\in \Sigma^k$, we write $$
\Phi_{\mathbf{w}}=\Phi_{w_1}\circ \Phi_{w_2}\circ \ldots\circ \Phi_{w_k}.
$$

Suppose that $J=[0,1]^2\subset \mathbb{R}^{2}$. For $k=1,2,\ldots$, let $\{\Phi_w\}_{w\in\mathcal{D}_k }$ be the self-affine IFS as in~\eqref{eq:Sk}.
For each $\mathbf{w}\in \Sigma^k$, the set $J_{\mathbf{w}}$ is a geometrical affine copy to
$J$, i.e., there exists an affine mapping $\Phi_{\mathbf{w}}:\mathbb{R}%
^{2}\rightarrow \mathbb{R}^{2}$ such that $J_{\mathbf{w}}=\Phi_{\mathbf{w}}(J)$.
The non-empty compact set
\begin{equation}\label{attractor}
E=\bigcap\nolimits_{k=1}^{\infty}\bigcup\nolimits_{\mathbf{w}\in \Sigma^{k}} J_{\mathbf{w}}
\end{equation}
is called a \textit{self-affine Moran set or self-affine set with Moran construction} $\{(n_k,m_k,\mathcal{D}_k)\}_{k=1}^\infty$. For all $\mathbf{w}\in \Sigma^{k}$, the elements $J_{\mathbf{w}}$
are called \textit{\ $k$th-level basic sets} of $E$, see Figure~\ref{figsm} for the first three levels.

Note that this  may also be regarded as a generalization of Moran fractals where only similarity contractions are used in the construction, see~\cite{Moran,Wen01}.  In~\cite{GM22}, the authors studies a special case of these sets where they require that $n_k\geq m_k$ for all $k>0$, and they provided  the Assouad, packing and box-counting dimensions of the sets. They also obtained the Hausdorff dimension formula under some strong technique assumptions.  In this paper, we are interested in investigating the dimension properties of measures supported on self-affine Moran  sets, and we also provide the Assouad, packing and box-counting dimension formulas of sets which extend the conclusions in~\cite{GM22}. Furthermore, we study the lower dimension of the self-affine sets which has not been studied,  and this conclusion completes the dimension formulas in inequality~\eqref{dim_ineq}.

\begin{figure}[h]
\centering
\includegraphics[width=0.8\textwidth]{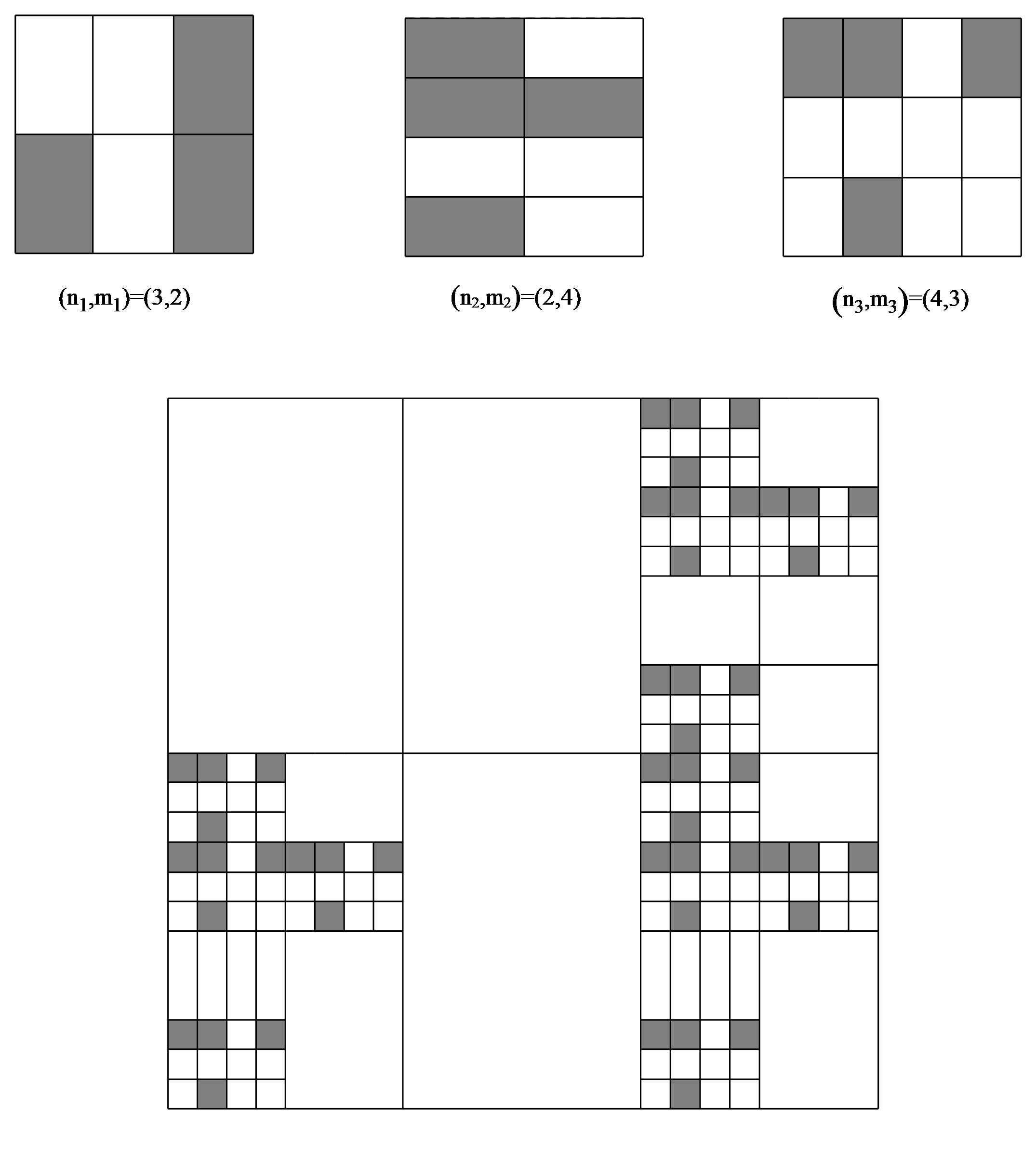}
\caption{Self-affine Moran constructed to Level 3, where $\D_1=\{(0,0),(2,0),(2,1)\}$, $\D_2=\{(0,0),(0,2),(0,3),(1,2)\}$ and $
\D_3=\{(0,2),(1,0),(1,2),(3,2)\}.$}\label{figsm}
\end{figure}

Let $\Pi: \Sigma^\infty \rightarrow \mathbb{R}^2$ be the projection given by
$$
\Pi(\mathbf{w})= \sum_{k=1}^{\infty} \mathrm{diag} \left(\prod_{h=1}^{k} n_h^{-1}, \prod_{h=1}^{k} m_h^{-1}\right)w_k.
$$
Then the self-affine Moran set $E$ is the image of $\Pi$, i.e. $E=\Pi(\Sigma^\infty)$.  Note that the projection $\Pi$ is surjective.

Let $\mathcal{P}_k$ denote the collection of  all probability vectors on $\D_k$, and $\mathcal{P}=\prod_{k=1}^\infty \mathcal{P}_k$.
Given $\bp=(\bp_k)_{k=1}^\infty\in\mathcal{P}$, where $\bp_k=\big(p_k(ij)\big)_{(i,j)\in \mathcal{D}_k}\in\mathcal{P}_k$ is a probability vector. For each $\mathbf{w}=w_{1} w_2 \cdots w_{k}\in \Sigma^k$, we write
\begin{equation}\label{nu}
   \nu_\bp([\mathbf{w}])=p_{\mathbf{w}}=p_1(w_{1}) p_2(w_2) \cdots p_k(w_{k}).
\end{equation}
Then $\nu_\bp$ is a Borel measure on $\Sigma^\infty$. It is clear that
\begin{equation}\label{projmu}
  \mu_\bp (A)=\nu_\bp(\Pi^{-1}A)
\end{equation}
is a Borel probability measure on $E$, and we call it a \textit{self-affine Moran measure} on $E$.

 For each $k>0$, we write that, for $w=(i,j)\in \mathcal{D}_k$,
$$
q_k(w)=q_k(j)=\sum_{(i,j)\in \mathcal{D}_k} p_k(i,j),\qquad \widehat{q}_k(w)=\widehat{q}_k(i)=\sum_{(i,j)\in \mathcal{D}_k} p_k(i,j).
$$

For each $\delta>0$, let $k=k(\delta)$ be the unique integer satisfying
\begin{equation}\label{def_k}
 \qquad\frac{1}{m_1}\frac{1}{m_2}\ldots \frac{1}{m_k}\leq \delta<\frac{1}{m_1}\frac{1}{m_2}\ldots \frac{1}{m_{k-1}}.
\end{equation}
Note that if there is no integer satisfying above equation, we always set $k=1$. For each given integer $k$, let $l=l(k)$ be the unique integer satisfying
\begin{equation} \label{def_l}
\frac{1}{n_1}\frac{1}{n_2}\ldots \frac{1}{n_l}\leq \frac{1}{m_1}\frac{1}{m_2}\ldots \frac{1}{m_{k}}<\frac{1}{n_1}\frac{1}{n_2}\ldots \frac{1}{n_{l-1}}.
\end{equation}
We sometimes write $l(\delta)$ for $l(k)$ if $k=k(\delta)$ is given by~\eqref{def_k}. If there is no ambiguity in the context, we just write $l$ instead of $l(k)$ for simplicity.

For each $\delta>0$ and every $\mathbf{w}=w_1w_2\ldots w_n\ldots \in \Sigma^\infty$, where $w_n=(i_n,j_n)$,  we write
$$
U(\delta, \mathbf{w})=\Big\{\mathbf{v}=v_1v_2\ldots v_n \ldots\in\Sigma^\infty:
\begin{array}{ll}
    i_n=i_n',&n=1,\ldots,l(\delta), \\
    j_n=j_n',&n=1,\ldots, k(\delta),
\end{array}
v_n=(i_n',j_n')\Big\},
$$
and we write $\mathcal{U}_\delta$ for the collection of all such sets, i.e.
$$
\mathcal{U}_\delta=\{U(\delta, \mathbf{w}): \bw\in \Sigma^\infty\}.
$$
We write
\begin{equation}\label{appsquare}
\mathcal{S}_\delta=\{\Pi(U): U\in \mathcal{U}_\delta\}.
\end{equation}
The elements $S$ of $\mathcal{S}_\delta$ are called the \textit{$\delta$-approximate squares}. The measure distributed on approximate squares is essential in finding the dimensions of sets and measure.

Let $\nu_\bp$ and $\mu_\bp$ be the  measures given by \eqref{nu} and \eqref{projmu}. Given $\delta>0$, for each $U(\delta, \mathbf{w})\in \mathcal{U}_\delta$, we have that
\begin{equation}\label{nuas}
  \nu_\bp(U(\delta, \mathbf{w}))=\left\{
  \begin{array}{lcl}
    p_1(w_1)\ldots p_l(w_l) q_{l+1}(w_{l+1})\ldots q_k(w_k),  & \  & l \leq k, \\
    p_1(w_1)\ldots p_k(w_k)\widehat{q}_{k+1}(w_{k+1}) \cdots \widehat{q}_l(w_l), & \  & l > k.
  \end{array}
  \right.
\end{equation}
where  $k=k(\delta)$ and $l=l(\delta)$ are given by \eqref{def_k} and \eqref{def_l} .
For each $S(\delta, x)\in \mathcal{S}_\delta$ where $x\in S(\delta,x)\cap E$, there exists a sequence $\mathbf{w}$ such that $\Pi(\mathbf{w})=x$ and $\Pi(U(\delta, x))= S(\delta, x)$. Then
\begin{equation}\label{muas}
  \mu_\bp(S(\delta, x))=  \nu_\bp(U(\delta, \mathbf{w})).
\end{equation}

  Approximate squares are an essential tool in studying self-affine fractals, see~\cite{Baran07, Bedfo84, Fraser21, LalGa92,McMul84}, and  we may also apply  this tool to study the dimensions of the self-affine Moran sets and self-affine Moran measures.

\section{ Main Results}\label{sec_SCMR}

In this section, we state our main conclusions. Let
\begin{equation}\label{UB}
  N^+=\sup\{n_k, m_k: k=1,2,\ldots\}.
\end{equation}
We always assume that $N^+$ is finite in the paper. Given $\bp \in \mathcal{P}$. For each integer $k>0$, the \textit{ $k$th entropy} is defined as
\begin{equation}\label{def_Hp}
  H_k(\bp)= \left\{
  \begin{array}{lcc}
  \sum_{i=1}^{l} \sum_{w\in \D_i} p_i(w)\log p_i(w) + \sum_{i=l+1}^{k} \sum_{w\in \D_i} p_i(w)\log q_i(w) & & l\leq k;\\
  \sum_{i=1}^{k} \sum_{w\in \D_i} p_i(w)\log p_i(w) + \sum_{i=k+1}^{l} \sum_{w\in \D_i} p_i(w)\log \widehat{q}_i(w) & & l> k,
  \end{array}\right.
\end{equation}
where $l=l(k)$ is given by~\eqref{def_l}.

Since the local dimension of self-affine Moran measures does not exist,  it leads to the non-existence of entropy dimension. First we  gives formulas of the upper and lower entropy dimensions by using $k$th entropy.
\begin{thm}\label{dime}
Let $E$ be the self-affine Moran set defined by ~\eqref{attractor} with  $N^+<\infty$. Given $\bp \in \mathcal{P}$, let $\mu_\bp$ be the self-affine Moran Measure defined by~\eqref{projmu}. Then
\begin{eqnarray*}
\overline{\dim}_e \mu_\bp &=& \limsup_{k \to \infty} \frac{H_k(\bp)}{\sum_{i=1}^{k} \log m_i};\\
\underline{\dim}_e \mu_\bp &=& \liminf_{k \to \infty} \frac{H_k(\bp)}{\sum_{i=1}^{k} \log m_i}.
\end{eqnarray*}
\end{thm}

To study the Hausdorff dimension of the self-affine Moran measures, it often requires certain separation conditions. We will introduce two such conditions from geometric and measure aspects. Given $k>0$, the set $\D_k$ is centred if $i\notin \{0,N-1\}$ and $j\notin \{0,M-1\}$ for every $(i,j)\in\D_k$. We say $E$ satisfies the \textit{frequency separation condition (FSC) } if there exists $c>0$ such that
$$
\lim_{n\to\infty} \frac{\card\{k:\D_k \textit{ is centred for }  k=1,\ldots,n\}}{n} = c.
$$
Given $\bp \in \mathcal{P}$, we say the self-affine Moran measure $\mu_\bp$ satisfies the \textit{measure separation condition (MSC)} if there exists a constant $0<C<1$ such that for each $k>0$, $$\max\{q_k(0),q_k(m_k-1),\widehat{q}_k(0),\widehat{q}_k(n_k-1)\}<C.$$

Next we show that the dimension formulas hold under either of FSC and MCS.
\begin{thm}\label{dim_murm}
Let $E$ be the self-affine Moran set defined by ~\eqref{attractor} with  $N^+<\infty$. Given $\bp \in \mathcal{P}$, let $\mu_\bp$ be the self-affine Moran Measure defined by~\eqref{projmu}. Suppose that either $E$ satisfies FSC or $\mu_\bp$ satisfies MSC. Then
\begin{eqnarray*}
\hdd \mu_\bp&=&\liminf_{k \to \infty} \frac{H_k(\bp)}{\sum_{i=1}^{k} \log m_i}; \\
\pkd  \mu_\bp&=&\limsup_{k \to \infty} \frac{H_k(\bp)}{\sum_{i=1}^{k} \log m_i}.
\end{eqnarray*}
\end{thm}
In section~\ref{dimHP}, the FSC is  replaced by a weaker condition, called boundary separation condition, see Theorem~\ref{thm_muHP} ,  and the Hausdorff and packing dimensions of measures are studies under the weak condition. Such geometric separation conditions are also useful to study the dimensions of sets.

It would be ideal that the supreme dimension of self-affine Moran measures equals the dimension of the sets, but we only obtain the equality under geometric separation conditions in the following special case, see Corollary~\ref{cor_BSC} as well.
\begin{cor}\label{cor_dim}
Let $E$ be the self-affine Moran set defined by ~\eqref{attractor} with  $N^+<\infty$.  Suppose that  $E$ satisfies FSC, and for all $k>0$, $n_k\geq m_k$, and $r_k(j)=c_k$ for all $j$ such that $r_k(j)\neq 0$. Then there exists $\bp\in \mathcal{P}$ such that
\begin{eqnarray*}
\hdd \mu_\bp&=&\max\{\hdd \mu_{\bp'} : \bp'\in \mathcal{P}\}=\hdd E;\\
\pkd \mu_\bp&=&\max\{\pkd \mu_{\bp'} : \bp'\in \mathcal{P}\}=\pkd E.
\end{eqnarray*}

\end{cor}

Next, we state  our conclusions on the dimension of self-affine Moran sets. We write
\begin{eqnarray*}
  r_k(j)&=&\operatorname{card}\{i\colon (i,j)\in \mathcal{D}_k \textrm{ for each } j\}, \\
  r_k^+&=&\max\{r_k(j): j=0,1, \ldots , m_k-1\}, \\
  r_k^-&=&\min\{r_k(j): r_k(j)\neq 0, j=0,1, \ldots , m_k-1\}, \\
  \widehat{r}_k(i)&=&\operatorname{card}\{j\colon (i,j)\in \mathcal{D}_k \textrm{ for each } i\}, \\
  \widehat{r}_k^+&=&\max\{\widehat{r}_k(i): i=0,1, \ldots , n_k-1\},\\
  \widehat{r}_k^-&=&\min\{\widehat{r}_k(i): \widehat{r}_k(i)\neq 0, i=0,1, \ldots , n_k-1\} \\
  s_k&=&\operatorname{card}\{j\colon (i,j)\in \mathcal{D}_k \textrm{ for some } i\},  \\
\widehat{s}_k&=&\operatorname{card}\{i\colon (i,j)\in \mathcal{D}_k \textrm{ for some } j\}.
\end{eqnarray*}
\iffalse Note that
$$
r_k=\sum_{i=0}^{n_k-1}  \widehat{r}_k(i)=\sum_{j=0}^{m_k-1} r_k(j).
$$
\fi

For each integer $k>0$, let $l=l(k)$ be given by~\eqref{def_l}, and we write
\begin{equation}\label{def_Nlk}
  N_{l,k}(E) =\left\{
  \begin{array}{lcc}
  r_1\ldots r_l s_{l+1} \ldots s_k, & & l \leq k,\\
  r_1\ldots r_k s_{k+1} \ldots s_l, & & l > k.
  \end{array} \right.
\end{equation}

We write
\begin{eqnarray}\label{}
  d^*=\limsup_{k\to \infty} \frac{N_{l,k}(E)}{\log m_1 \ldots m_k},&&
  d_*=\liminf_{k\to \infty} \frac{N_{l,k}(E)}{\log m_1 \ldots m_k}.
\end{eqnarray}
The box dimension and packing dimension of $E$ are bounded by $d^*$ and $d_*$.
\begin{thm}\label{thmB}
Let $E$ be the self-affine Moran set defined by ~\eqref{attractor} with  $N^+<\infty$. The packing dimension,  upper Box dimension and lower Box dimension of $E$ are given  by
$$
\pkd E=\ubd E=d^*; \qquad \lbd E=d_*.
$$
\end{thm}
The proof of the theorem is similar to the one of ~\cite[Theorem2.1]{GM22}, and we omit it.

Finally,we state the conclusions on lower and Assouad dimensions for self-affine Moran sets.  For  integers $k$ and $k'$ such that $k'>k>1$, let $l=l(k)$ and $l'=l'(k')$ be given by~\eqref{def_l}, we write
\begin{equation}\label{def_Nkk-}
  N_{k,k'}^-(E)= \left\{
  \begin{array}{lcc}
    r_{l+1}^- \ldots r_{l'}^- s_{k+1} \ldots s_{k'}, &  & l < l'\leq k <k', \\
    r_{l+1}^- \ldots r_k^- r_{k+1} \ldots r_{l'} s_{l'+1} \ldots s_{k'}, &  & l \leq k < l' \leq k', \\
    r_{l+1}^- \ldots r_k^- r_{k+1} \ldots r_{k'} \widehat{s}_{k'+1} \ldots \widehat{s}_{l'}, &  & l \leq k < k' \leq l' , \\
    \widehat{r}_{k+1}^- \ldots \widehat{r}_l^- r_{l+1} \ldots r_{l'} s_{l'+1} \ldots s_{k'}, &  & k \leq l < l' \leq k' , \\
    \widehat{r}_{k+1}^- \ldots \widehat{r}_l^- r_{l+1} \ldots r_{k'} \widehat{s}_{k'+1} \ldots \widehat{s}_{l'}, &  & k \leq l < k' \leq l', \\
    \widehat{r}_{k+1}^- \ldots \widehat{r}_{k'}^- \widehat{s}_{l+1} \ldots \widehat{s}_{l'}, &  & k < k' \leq l < l'.
  \end{array} \right.
\end{equation}
and
\begin{equation}\label{def_Nkk+}
  N_{k,k'}^+(E)= \left\{
  \begin{array}{lcc}
    r_{l+1}^+ \ldots r_{l'}^+ s_{k+1} \ldots s_{k'}, &  & l < l'\leq k <k', \\
    r_{l+1}^+ \ldots r_k^+ r_{k+1} \ldots r_{l'} s_{l'+1} \ldots s_{k'}, &  & l \leq k < l' \leq k', \\
    r_{l+1}^+ \ldots r_k^+ r_{k+1} \ldots r_{k'} \widehat{s}_{k'+1} \ldots \widehat{s}_{l'}, &  & l \leq k < k' \leq l' , \\
    \widehat{r}_{k+1}^+ \ldots \widehat{r}_l^+ r_{l+1} \ldots r_{l'} s_{l'+1} \ldots s_{k'}, &  & k \leq l < l' \leq k' , \\
    \widehat{r}_{k+1}^+ \ldots \widehat{r}_l^+ r_{l+1} \ldots r_{k'} \widehat{s}_{k'+1} \ldots \widehat{s}_{l'}, &  & k \leq l < k' \leq l', \\
    \widehat{r}_{k+1}^+ \ldots \widehat{r}_{k'}^+ \widehat{s}_{l+1} \ldots \widehat{s}_{l'}, &  & k < k' \leq l < l'.
  \end{array} \right.
\end{equation}

The following theorem shows that the lower and Assouad dimension of $E$ may be estimated by $N_{k,k'}^-(E)$ and $N_{k,k'}^+(E)$.
\begin{thm}\label{thmL}
Let $E$ be the self-affine Moran set defined by ~\eqref{attractor} with  $N^+<\infty$.  The lower dimension and the Assouad dimension  of $E$ are given by
$$
\lwd E = \lim_{m \to \infty} \inf_{k}\left\{ \frac{\log N_{k,k+m}^-(E)}{\log m_{k+1} \ldots m_{k+m}}\right\} ;
$$
$$
\asd E = \lim_{m \to \infty} \sup_{k}\left\{ \frac{\log N_{k,k+m}^+(E)}{\log m_{k+1} \ldots m_{k+m}}\right\}.
$$

\end{thm}

\section{Entropy dimensions of self-affine Moran measures}\label{sec_dime}
To investigate the entropy dimension, we need the following well-known inequality.
\begin{lem}\label{lem_ent}
The function $f:[0,\infty) \to R$ defined by
$$
f(x)=\left\{\begin{array}{ll} 0 & \textit{if  } x=0, \\ -x\log x & \textit{if  } x\neq 0        \end{array} \right.
$$
is strictly concave, and for all $x_1,\ldots,x_N\geq 0$,
$$
f\left(\sum_{i=1}^N x_i\right) \leq \sum_{i=1}^N f(x_i) \leq f\left(\sum_{i=1}^N x_i\right) + \left(\sum_{i=1}^N x_i\right)\log N.
$$
\end{lem}

Let $f$ be the function defined in Lemma~\ref{lem_ent}. The lower and upper entropy dimension may be rewritten as
\begin{eqnarray*}
\underline{\dim}_e \mu_\bp &=& \liminf_{n\to\infty} \frac{\sum_{Q\in\mathcal{M}_n} f(\mu_\bp(Q))}{\log 2^{-n}};\\
\overline{\dim}_e \mu_\bp &=& \limsup_{n\to\infty} \frac{\sum_{Q\in\mathcal{M}_n} f(\mu_\bp(Q))}{\log 2^{-n}}.
\end{eqnarray*}

\begin{proof}[Proof of Theorem~\ref{dime}]
For each $\delta>0$, let $n$ be the integer such that $2^{-n} \leq \delta < 2^{-n+1}$. Then for each $Q\in\mathcal{M}_n$, it intersects at most $C_1= 4(N^+)^3$ approximate squares of $\mathcal{S}_\delta$, and for each $S\in\mathcal{S}_\delta$, it intersects at most $3^2$ cubes in $\mathcal{M}_n$.  Therefore by Lemma~\ref{lem_ent}, It follows that for each $Q\in \mathcal{M}_n$
\begin{equation}\label{fvs}
f(\mu_\bp(Q)) \leq \sum_{S\in\mathcal{S}_\delta} f(\mu_\bp(S\cap Q)) \leq f(\mu_\bp(Q))+ (\log C_1)\mu_\bp(Q)
\end{equation}
and
for each $S\in\mathcal{S}_\delta$,
\begin{equation}\label{fvq}
f(\mu_\bp(S)) \leq \sum_{Q\in \mathcal{M}_n} f(\mu_\bp(S\cap Q)) \leq f(\mu_\bp(S))+(2\log 3)\mu_\bp(S).
\end{equation}

Summing up ~\eqref{fvq} and ~\eqref{fvs} respectively, we obtain that
\begin{eqnarray*}
\sum_{Q\in\mathcal{M}_n} f(\mu_\bp(Q)) &\leq& \sum_{Q\in\mathcal{M}_n} \sum_{S\in\mathcal{S}_\delta}f(\mu_\bp(S\cap Q)) \leq \sum_{Q\in\mathcal{M}_n} f(\mu_\bp(Q))+\log C_1; \\
\sum_{S\in\mathcal{S}_\delta} f(\mu_\bp(S)) &\leq& \sum_{S\in\mathcal{S}_\delta} \sum_{Q\in\mathcal{M}_n} f(\mu_\bp(Q\cap S)) \leq \sum_{S\in\mathcal{S}_\delta} f(\mu_\bp(S))+2\log 3.
\end{eqnarray*}
It follows that
\begin{equation}\label{fp}
|\sum_{S\in\mathcal{S}_\delta} f(\mu_\bp(S)) - \sum_{Q\in\mathcal{M}_n} f(\mu_\bp(Q))| \leq 2\log 3 +\log C_1.
\end{equation}

Let $l=l(\delta)$ and $k=k(\delta)$ be given by $\eqref{def_l}$ and $\eqref{def_k}$. For  $l\leq k$, by induction, it follows that
\begin{eqnarray*}
\sum_{S\in\mathcal{S}_\delta} f(\mu_\bp(S)) &=& \sum_{U(\delta,\bw)\in \mathcal{U}_\delta} p_1(w_1)\ldots p_l(w_l) q_{l+1}(w_{l+1})\ldots q_k(w_k)\\
&&\hspace{3cm} \log p_1(w_1)\ldots p_l(w_l) q_{l+1}(w_{l+1})\ldots q_k(w_k) \\
&=& \sum_{\bw\in \Sigma^k} p_1(w_1)\ldots p_l(w_l) q_{l+1}(w_{l+1})\ldots q_k(w_k) \\
&&\hspace{3cm} \log p_1(w_1)\ldots p_l(w_l) \\
&&\hspace{0cm}+\sum_{\bw\in \Sigma^k} p_1(w_1)\ldots p_l(w_l) q_{l+1}(w_{l+1})\ldots q_k(w_k) \\
&&\hspace{3cm} \log  q_{l+1}(w_{l+1})\ldots q_{k}(w_{k})\\
&=& \sum_{i=1}^{l} \sum_{w\in \D_i} p_i(w)\log p_i(w) + \sum_{i=l+1}^{k} \sum_{w\in \D_i} p_i(w)\log q_i(w) .
\end{eqnarray*}
For $l> k$, similarly, we have that
\begin{eqnarray*}
\sum_{S\in\mathcal{S}_\delta} f(\mu_\bp(S)) &=& \sum_{U(\delta,\bw)\in \mathcal{U}_\delta} p_1(w_1)\ldots p_k(w_k) \widehat{q}_{k+1}(w_{k+1})\ldots \widehat{q}_l(w_l)\\
&&\hspace{3cm} \log p_1(w_1)\ldots p_k(w_k) \widehat{q}_{k+1}(w_{k+1})\ldots \widehat{q}_l(w_l) \\
&=& \sum_{i=1}^{k} \sum_{w\in \D_i} p_i(w)\log p_i(w) + \sum_{i=k+1}^{l} \sum_{w\in \D_i} p_i(w)\log \widehat{q}_i(w) .
\end{eqnarray*}
Hence, by~\eqref{def_Hp}, we obtain that
$$
\sum_{S\in\mathcal{S}_\delta} f(\mu_\bp(S)) = H_k(\bp).
$$
Combining this with \eqref{fp}, we have that
\begin{eqnarray*}
\overline{\dim}_e \mu_\bp &=& \limsup_{n\to\infty} \frac{\sum_{Q\in\mathcal{M}_n} f(\mu_\bp(Q))}{\log 2^{-n}}  \\
&=& \limsup_{\delta\to 0} \frac{\sum_{S\in\mathcal{S}_\delta} f(\mu_\bp(S))}{\log \delta} \\
&=& \limsup_{k \to \infty} \frac{H_k(\bp)}{\sum_{i=1}^{k} \log m_i}.
\end{eqnarray*}

By the same argument, we have that
\begin{eqnarray*}
\underline{\dim}_e \mu_\bp &=& \liminf_{n\to\infty} \frac{\sum_{Q\in\mathcal{M}_n} f(\mu_\bp(Q))}{\log 2^{-n}}  \\
&=& \liminf_{k \to \infty} \frac{H_k(\bp)}{\sum_{i=1}^{k} \log m_i},
\end{eqnarray*}
and we complete the proof.

\end{proof}

\section{Hausdorff and packing dimensions of self-affine Moran measures } \label{dimHP}

In this section, we study the Hausdorff and packing dimensions of self-affine Moran measures  in a weak condition.

Given $k>0$, the set $\D_k$ is left(right, bottom, top) empty if $i\neq 0(i\neq n_k-1, j\neq 0, j\neq m_k-1)$. We say $E$ is left separated if
$$
\lim_{n\to\infty} \frac{\card\{k:\D_k \textit{ is left empty for }  k=1,\ldots,n\}}{n} = c_L .
$$
Similarly, we may define $E$ is right separated, top separated and bottom separated where the limits are denoted by $c_R, c_T, c_B$ respectively.  If $E$ is left, right, bottom and top separated, we say $E$ satisfies the \textit{boundary separation condition (BSC)}.

Since BSC is  weaker than  FSC,  we prove the dimension formulas of self-affine Moran measures under BSC.
\begin{thm}\label{thm_muHP}
Let $E$ be the self-affine Moran set defined by ~\eqref{attractor} with  $N^+<\infty$. Given $\bp \in \mathcal{P}$, let $\mu_\bp$ be the self-affine Moran Measure defined by~\eqref{projmu}. Suppose that either $E$ satisfies BSC or $\mu_\bp$ satisfies $MSC$. Then
\begin{eqnarray*}
\hdd \mu_\bp&=&\liminf_{k \to \infty} \frac{H_k(\bp)}{\sum_{i=1}^{k} \log m_i}; \\
\pkd  \mu_\bp&=&\limsup_{k \to \infty} \frac{H_k(\bp)}{\sum_{i=1}^{k} \log m_i}.
\end{eqnarray*}
\end{thm}
\begin{cor}\label{cor_BSC}
Let $E$ be an arbitrary self-affine Moran set defined by ~\eqref{attractor}.Suppose that  $E$ satisfies BSC, and for all $k>0$, $n_k\geq m_k$, and $r_k(j)=c_k$ for all $j$ such that $r_k(j)\neq 0$. Then there exists $\bp\in \mathcal{P}$ such that
\begin{eqnarray*}
\hdd \mu_\bp&=&\max\{\hdd \mu_{\bp'} : \bp'\in \mathcal{P}\}=\hdd E;\\
\pkd \mu_\bp&=&\max\{\pkd \mu_{\bp'} : \bp'\in \mathcal{P}\}=\pkd E.
\end{eqnarray*}

\end{cor}

To study the dimensions of self-affine Moran measures, we need  a version of the law of large numbers.  For the readers' convenience, we cite it here, see, for example, \cite[Corollary A.8]{BYQJF} for details.

\begin{thm}\label{lln}
Let $\{X_n\}_{n=1}^\infty $ be a sequence of random variables which are bounded in $L^2$ and such that
$$
\mathbf{E} (X_n | X_1, \ldots, X_{n-1})=0
$$
for all $ n \geq 1.$
Then the sequence $\frac{1}{n}\sum_{i=1}^n X_i$ converges to 0 almost surely and in $L^2$.
\end{thm}

To estimate the Hausdorff dimension, we need the following well-known fact, which is often called Frostman’s Lemma, see~\cite{Falco97}.
\begin{lem}\label{lem_frost}
 Let $\mu$ be a finite Borel measure in $\R^d$.\\
(1) If $\liminf_{r\to 0} \frac{\log \mu(B(x,r)}{\log r} \geq s$ for $\mu$-almost every $x$, then $\hdd \mu\geq s.$\\
(2) If $\liminf_{r\to 0} \frac{\log \mu(B(x,r)}{\log r} \leq s$ for $\mu$-almost every $x$, then $\hdd  \mu\leq s.$\\
(3) If $\limsup_{r\to 0} \frac{\log \mu(B(x,r)}{\log r} \geq s$ for $\mu$-almost every $x$, then $\pkd  \mu\geq s.$\\
(4) If $\limsup_{r\to 0} \frac{\log \mu(B(x,r)}{\log r} \leq s$ for $\mu$-almost every $x$, then $\pkd  \mu\leq s.$
\end{lem}

\begin{proof}[Proof of Theorem~\ref{thm_muHP}]
First, we show that the conclusion holds for BSC. Given $\bw=w_1w_2\ldots w_k\ldots \in \Sigma^\infty$, since $(\log p_k(w_k))_{k\in \mathbb{N}}$ is a sequence of independent random variables, their variances are uniformly bounded by
$$
\mathbf{Var} (\log p_k(w_k)) \leq (N^+)^2 \max_{x\in [0.1]} x\log^2 x.
$$
By Theorem~\ref{lln}, we have
$$
-\sum_{k=1}^N \log p_k(w_k) = \sum_{k=1}^N \sum_{w\in\D_k} p_k(w)\log p_k(w)+o(N),
$$
almost surely.
Similarly, the following equalities hold
$$
-\sum_{k=1}^N \log q_k(w_k) = \sum_{k=1}^N \sum_{w\in\D_k} p_k(w)\log q_k(w)+o(N);
$$
$$
-\sum_{k=1}^N \log \widehat{q}_k(w_k) = \sum_{k=1}^N \sum_{w\in\D_k} p_k(w)\log \widehat{q}_k(w)+o(N),
$$
almost surely.

For each integer $k>0$, set $\delta=(m_1\ldots m_k)^{-1}$, and write $U(k,\bw)=U(\delta,\bw)$ and $S_k(\bw)=\Pi(U(k,\bw))$. By \eqref{nuas}, we have that for  $k\geq l$,
\begin{eqnarray*}
\log \nu_\bp(U(k,\bw))&=& \sum_{i=1}^l p_i(w_i)+ \sum_{i=l+1}^k q_i(w_i)\\
&=& \sum_{i=1}^l \sum_{w\in\D_k} p_k(w)\log p_k(w) + \sum_{i=l+1}^k \sum_{w\in\D_k} p_k(w)\log q_k(w) + o(k),
\end{eqnarray*}
and for $k<l$,
\begin{eqnarray*}
\log \nu_\bp(U(k,\bw))&=& \sum_{i=1}^k p_i(w_i)+ \sum_{i=k+1}^l q_i(w_i)\\
&=& \sum_{i=1}^k\sum_{w\in\D_k}  p_k(w)\log p_k(w) + \sum_{i=k+1}^l\sum_{w\in\D_k}  p_k(w)\log \widehat{q}_k(w) + o(k),
\end{eqnarray*}
almost surely. Hence, by \eqref{muas} and \eqref{def_Hp}, it follows that
\begin{equation}\label{mu_HO}
\log \mu_\bp(S_k(\bw)) = \log \nu_\bp(U(k,\bw)) = H_k(\bp)+o(k),
\end{equation}
almost surely.

Fix $\varepsilon>0$, and let $\xi=\frac{2\varepsilon}{1-\varepsilon}$. It is clear that  $\xi\rightarrow0$ as $\varepsilon$ tends to $0$. Since $E$ satisfies the separation condition,  there exists $K_0>0$ such that for $k>K_0$,
\begin{eqnarray}\label{Dh}
&&\card\{h:\D_h \textit{ is left empty for  } (1-\xi)k<h<k\}\geq 1\\
&&\card\{h:\D_h \textit{ is right empty for  } (1-\xi)k<h<k\}\geq 1 \nonumber\\
&&\card\{h:\D_h \textit{ is top empty for  } (1-\xi)k<h<k\}\geq 1 \nonumber\\
&&\card\{h:\D_h \textit{ is bottom empty for  } (1-\xi)k<h<k\}\geq1 \nonumber
\end{eqnarray}

For sufficiently small $\rho>0$, let $k$ be the integer such that
$$
\prod_{i=1}^k m_i \leq \rho < \prod_{i=1}^{k-1} m_i \leq (N^+)^{-1} \prod_{i=1}^k m_i.
$$

Let $l=l(k)$ be given by \eqref{def_l}. Setting
$$
k'=k+1,\quad k''=\min\{(1-\xi)k,k((1-\xi)^2 l)\},
$$
where $k((1-\xi)^2 l)$ denotes the largest integer $\beta$ such that $l(\beta) \leq (1-\xi)^2 l$. Then, by \eqref{Dh}, we have that
\begin{eqnarray*}
&&\card\{h:\D_h \textit{ is left for  } l''(k'')<h<(1-\xi)l\}\geq 1, \\
&&\card\{h:\D_h \textit{ is right for  } l''(k'')<h<(1-\xi)l\}\geq 1,\\
&&\card\{h:\D_h \textit{ is top empty for  } k''<h<k\}\geq 1, \\
&&\card\{h:\D_h \textit{ is bottom empty for  } k''<h<k\}\geq 1.
\end{eqnarray*}

Next we show that the distance from $\Pi(\bw)$ to the each side of $S_{k''}(\bw)$ is greater than $\rho$. We first  consider the distance from $\Pi(\bw)$ to the left side of $S_{k''}(\bw)$. Let $l_0$ an integer satisfy $l''(k'')<l_0<(1-\xi)l$. Then the distance from $\Pi(\bw)$ to the left side of $S_{k''}(\bw)$ is no less than $(n_1\ldots n_{l_0})^{-1}$. Since $l$ is sufficiently large, $\xi l \geq \frac{\log (N^+)^2}{\log 2}$, It is clear that
$$
(n_1\ldots n_{l_0})^{-1} \geq 2^{\xi l} (n_1 \ldots n_l)^{-1} \geq (N^+)^2 (n_1 \ldots n_l)^{-1} \geq \rho.
$$
Hence  the distance from $\Pi(\bw)$ to the left side of $S_{k''}(\bw)$ is greater than $\rho$. For  the distance from $\Pi(\bw)$ to the bottom side of $S_{k''}(\bw)$. similarly, we may find an integer $k_0$, $k''<k_0<k$ such that $\D_{k_0}$ is bottom empty. Then the distance from $\Pi(\bw)$ to the bottom side of $S_{k''}(\bw)$ is no less than $(m_1\ldots m_{k_0})^{-1}$, which is greater than $\rho$.

Similarly, the distances from $\Pi(\bw)$ to the top and right sides of $S_{k''}(\bw)$ are greater than $\rho$ as well. This implies that $
B(\Pi(\bw),\rho) \subset S_{k''}(\bw).$  Since $(m_1 \ldots m_{k'})^{-1} < (m_1 \ldots m_{k})^{-1} \leq \rho$, we have $S_{k'}(\bw) \subset B(\Pi(\bw),\rho).$

Therefore, we obtain that
\begin{equation}\label{incBS}
S_{k'}(\bw) \subset B(\Pi(\bw),\rho) \subset S_{k''}(\bw).
\end{equation}
By \eqref{mu_HO}, immediately, we have that
\begin{eqnarray*}
\liminf_{k\to \infty} \frac{H_{k'}(\bp)+o(k)}{\sum_{i=1}^{k} \log m_i} &\geq& \liminf_{\rho\to 0} \frac{\log \mu_\bp(B(\Pi(\bw),\rho))}{\log \rho}
\geq \liminf_{k\to \infty} \frac{H_{k''}(\bp)+o(k)}{\sum_{i=1}^{k} \log m_i}
\end{eqnarray*}
almost surely.  Note that
$$
H_{k'}(\bp)\to H_{k}(\bp) , \qquad H_{k''}(\bp)\to H_{k}(\bp)
$$
as $\varepsilon$ tends to  $0$, they imply that
$$
\liminf_{\rho\to 0} \frac{\log \mu_\bp(B(\Pi(\bw),\rho))}{\log \rho}=\liminf_{k \to \infty} \frac{H_k(\bp)}{\sum_{i=1}^{k} \log m_i}.
$$
almost surely. By Lemma~\ref{lem_frost},  it follows that
$$
\dimh \mu_\bp =\liminf_{k \to \infty} \frac{H_k(\bp)}{\sum_{i=1}^{k} \log m_i}.
$$
Similarly, By~\eqref{incBS} and \eqref{mu_HO}, we have that
\begin{eqnarray*}
\limsup_{k\to \infty} \frac{H_{k'}(\bp)+o(k)}{\sum_{i=1}^{k} \log m_i} &\geq& \limsup_{\rho\to 0} \frac{\log \mu_\bp(B(\Pi(\bw),\rho))}{\log \rho}
\geq \limsup_{k\to \infty} \frac{H_{k''}(\bp)+o(k)}{\sum_{i=1}^{k} \log m_i}.
\end{eqnarray*}

By Lemma~\ref{lem_frost},  we have that
$$
\pkd \mu_\bp =\limsup_{k \to \infty} \frac{H_k(\bp)}{\sum_{i=1}^{k} \log m_i}.
$$

Next we prove that the conclusion holds for MSC.  For each integer $k>0$, we write
$$
A_k=\{x=\Pi(\mathbf{w})\in E: B\big(x,(m_1\ldots m_k)^{-1}e^{-\sqrt{k}}\big)\cap E \subset S_k(\mathbf{w})\}.
$$
Let $L_k$ be the collection of $x\in E$ such that the distance from $x$ to the bottom side of $S_k(\mathbf{w})$ is less than $(m_1\ldots m_k)^{-1}e^{-\sqrt{k}}$, where $x=\Pi(\mathbf{w})$, $\mathbf{w}=(i_1,j_1)\ldots(i_k,j_k)\ldots\in\Sigma^\infty$. It is clear that  $j_{k+1}= \ldots =j_{k+[\sqrt{k}/\log N^+]}=0$. Hence,  the measure of $L_k$ is  bounded by
$$
\mu_\bp(L_k)\leq q_{k+1}(0)\ldots q_{{k+[\sqrt{k}/\log N^+]}}(0) \leq C_1^{\sqrt{k}},
$$
where $C_1=C^{1/(\log N^+ +1)}<1$. We apply the similar argument to other three sides and obtain that
$$
\sum_{k=1}^\infty \mu_\bp(A_k^c) < 4 \sum_{k=1}^\infty C_1^{\sqrt{k}}<\infty.
$$
By Borel-Cantelli Lemma, it follows that
$$\mu_\bp(A_k^c\,i.o.)=0.
$$
Therefore, for $\mu_\bp$-almost all $x$, we have that
$$
\mu_\bp\Big(B\big(x, (m_1\ldots m_k)^{-1}e^{-\sqrt{k}}\big)\Big)\leq \mu_\bp(S_k(\mathbf{w})),
$$
for sufficiently large $k$.  For each $\rho>0$, there exists a unique integer $k$ such that
$$
(m_1\ldots m_{k+1})^{-1}e^{-\sqrt{k+1}}\leq \rho<(m_1\ldots m_k)^{-1}e^{-\sqrt{k}},
$$
which implies that $\mu_\bp(B(x,\rho))\leq \mu_\bp(S_k(x))$. Therefore, we have that
$$
\liminf_{\rho\to 0}\frac{\log \mu_\bp(B(x,\rho))}{\log \rho}\geq \liminf_{k\to \infty}\frac{\log \mu_\bp(S_k(\mathbf{w}))}{-\sum_{i=1}^{k} \log m_i} =\liminf_{k\to \infty}\frac{H_k(\bp)}{\sum_{i=1}^{k} \log m_i}.
$$

On the other hand, for each $\rho>0$, let $k$ be the integer such that
$$
\prod_{i=1}^k m_i < \rho \leq \prod_{i=1}^{k-1} m_i.
$$
Then for all $x\in E$, choose $\bw\in \Pi^{-1}(x)$, and we have $S_k(\bw)\subset B(x,\rho)$, which implies $\mu_\bp(B(x,\rho))\geq \mu_\bp(S_k(x))$. Therefore, by~\eqref{mu_HO}, we have that
$$
\liminf_{\rho\to 0}\frac{\log \mu_\bp(B(x,\rho))}{\log \rho}\leq \liminf_{k\to \infty}\frac{\log \mu_\bp(S_k(\mathbf{w}))}{-\sum_{i=1}^{k} \log m_i} =\liminf_{k\to \infty}\frac{H_k(\bp)}{\sum_{i=1}^{k} \log m_i}.
$$
It follows that
$$
\liminf_{\rho\to 0} \frac{\log \mu_\bp(B(x,\rho))}{\log \rho}=\liminf_{k \to \infty} \frac{H_k(\bp)}{\sum_{i=1}^{k} \log m_i}.
$$
almost surely. By Lemma~\ref{lem_frost},  the Hausdorff dimension of $\mu_\bp$ is given by
$$
\dimh \mu_\bp =\liminf_{k \to \infty} \frac{H_k(\bp)}{\sum_{i=1}^{k} \log m_i}.
$$

Similarly, for $\mu_\bp$-almost all $x$, we have that
$$
\limsup_{r\to 0}\frac{\log \mu_\bp(B(x,\rho))}{\log \rho}=\limsup_{k \to \infty} \frac{H_k(\bp)}{\sum_{i=1}^{k} \log m_i},
$$
and by Lemma~\ref{lem_frost},  the packing dimension of $\mu_\bp$ is given by
$$
\pkd \mu_\bp =\limsup_{k \to \infty} \frac{H_k(\bp)}{\sum_{i=1}^{k} \log m_i}.
$$
\end{proof}

\begin{proof}[proof of Corollary~\ref{cor_BSC}]
For each $k>0$, let $p_k(w)=\frac{1}{r_k}$ for all $w\in \D_k$. Since $r_k(j)=c_k$ for all $j$ such that $r_k(j)\neq 0$, we have $r_k=c_k s_k$, and it implies $q_k(w)=\frac{1}{s_k}$. By~\eqref{nuas},
\begin{eqnarray*}
\log \mu_\bp(S_k(\bw)) &=& \log \nu_\bp(U(k,\bw))\\
&=& \sum_{i=1}^{l} \log p_k(w) + \sum_{k=l+1}^{k} \log q_k(w)  \\
&=& -\sum_{i=1}^{l} \log r_k - \sum_{k=l+1}^{k} \log s_k
\end{eqnarray*}
for all $\bw\in \Sigma^\infty$ and $k>0$. By the same argument in Theorem~\ref{dim_murm}, we have that
$$\liminf_{\rho\to 0} \frac{\log \mu_\bp(B(x,\rho))}{\log \rho} = \liminf_{k\to \infty} \frac{\sum_{i=1}^{l} \log r_k + \sum_{k=l+1}^{k} \log s_k}{\sum_{i=1}^{k} \log m_i}.
$$
for all $x\in E$. Then by ~\cite[Proposition 2.3]{Falco97},
$$
\dimh E = \dimh \mu_\bp = \liminf_{k\to \infty} \frac{\sum_{i=1}^{l} \log r_k + \sum_{k=l+1}^{k} \log s_k}{\sum_{i=1}^{k} \log m_i}.
$$
Since $\dimh E = \sup\{\dimh \mu; \text{ for all Borel } \mu \text{ on } E \text{ such that } 0<\mu(E)<\infty\}$, we have that
$$
\hdd \mu_\bp=\max\{\hdd \mu_{\bp'} : \bp'\in \mathcal{P}\}=\hdd E.
%=\liminf_{k\to \infty} \frac{\sum_{i=1}^{l} \log r_k + \sum_{k=l+1}^{k} \log s_k}{\sum_{i=1}^{k} \log m_i}.
$$
Similarly, we have that
$$
\pkd E  = \pkd \mu_\bp = \limsup_{k\to \infty} \frac{\sum_{i=1}^{l} \log r_k + \sum_{k=l+1}^{k} \log s_k}{\sum_{i=1}^{k} \log m_i},
$$
and this implies that
$$
\pkd \mu_\bp=\max\{\pkd \mu_{\bp'} : \bp'\in \mathcal{P}\}=\pkd E.
$$
\end{proof}

\begin{proof}[proof of Theorem~\ref{dim_murm}]
Since $E$ satisfies the frequency separation condition, there exists $c>0$ such that
$$
\lim_{n\to\infty} \frac{\card\{k:\D_k \textit{ is centred for }  k=1,\ldots,n\}}{n} = c.
$$
It is clear that If $\D_k$ is centred, then $\D_k$ is left, right, top and bottom separated and  $c_L=c_R=c_T=c_B=c$, that is, FSC implies BSC. Hence $E$ satisfies boundary separation condition, and by Theorem~\ref{thm_muHP},  the conclusion holds
\end{proof}
\begin{proof}[proof of  Corollary~\ref{cor_dim}]
Since FSC implies BSC, by Corollary~\ref{cor_BSC}, the conclusion holds
\end{proof}

\section{Lower and Assouad dimensions of self-affine Moran sets}\label{sec_Ld}
In  this section, we give the proofs for the lower and Assouad dimension of self-affine Moran sets.

First, we show the connection between approximate squares and balls which is fundamental to our proofs.  For simplicity, we write
$$R_k=(m_1 \ldots m_k)^{-1}.$$
\begin{lem}
For every approximate square $S\in \mathcal{S}_k$, there exists $x\in S$ such that $B(x,(N^+)^{-3}R_k)\cap E \subset S$.
\end{lem}

\begin{proof}
Suppose that $S=\Pi(U)$ is an approximate square in $\mathcal{S}_k$ with $l\leq k$(for the case $l>k$, the conclusion follows by exchanging the roles of x and y axes), where
$$
U=\Big\{\mathbf{w}=w_1 w_2 \ldots w_n \ldots:\begin{array}{ll}
    i_n=i_n(S),&n=1,\ldots,l, \\
    j_n=j_n(S),&n=1,\ldots, k,
\end{array}
w_n=(i_n,j_n)\Big\}.
$$
Define
\begin{multline*}
\partial U=\Big\{\mathbf{w}=w_1 w_2 \ldots w_n \ldots:
i_{l+1}=i_{l+2}=0 \text{ or } i_{l+1}=n_{l+1}-1,i_{l+2}=n_{l+2}-1, \\
j_{k+1}=j_{k+2}=0 \text{ or } j_{k+1}=m_{k+1}-1,j_{k+2}=m_{k+2}-1,w_n=(i_n,j_n)\Big\}.
\end{multline*}
If there exists $\mathbf{w}\in U$ such that $\mathbf{w}\notin \partial U$, then by taking $x=\Pi(\mathbf{w})$, we have that $B(x,(N^+)^{-3}R_k)\cap E \subset S$. Otherwise $U\backslash\partial U$ is empty, then $\partial U$ is nonempty.

First suppose that there exists $\mathbf{w}\in \partial U$ with $i_{l+1}=i_{l+2}=0$(or $i_{l+1}=n_{l+1}-1,i_{l+2}=n_{l+2}-1$) and $j_{k+1}\notin \{0,m_{k+1}-1\}$ or $j_{k+2}/(m_{k+2}-1)\neq j_{k+1}/(m_{k+1}-1)$. It follows that if $k\geq l+1$ then $(n_{l+1}-1, j_{l+1}(S))\notin \D_{l+1}$(or $(0, j_{l+1}(S))\notin \D_{l+1}$), and if $k=l$ then $ \widehat{r}_l(n_{l+1}-1)=0$(or $\widehat{r}_l(0)=0$). Arbitrarily choose such $\mathbf{w}$ and let $x=\Pi(\mathbf{w})$, then $B(x,(N^+)^{-3}R_k)\cap E \subset S$ and the conclusion holds.

Otherwise, each $\mathbf{w}=w_1 w_2 \ldots w_n \ldots\in U, w_n=(i_n,j_n)$ satisfies $j_{k+1}=j_{k+2}=0$ or $j_{k+1}=m_{k+1}-1,j_{k+2}=m_{k+2}-1$. Without loss of generality, suppose that there exists $\mathbf{w}\in U$ with $j_{k+1}=j_{k+2}=0$. It follows that $r_k(m_{k+1}-1)=0$.

If there exists $\mathbf{w}\in U$ with $j_{k+1}=j_{k+2}=0$ and $i_{l+1}\notin \{0,n_{l+1}-1\}$ or $i_{l+2}/(n_{l+2}-1)\neq i_{l+1}/(n_{l+1}-1)$, then by taking such $\mathbf{w}$ and $x=\Pi(\mathbf{w})$, we have that $B(x,(N^+)^{-3}R_k)\cap E \subset S$ and the conclusion holds. Otherwise, for each $\mathbf{w}\in U$, $i_{l+1}=i_{l+2}=0$(or $i_{l+1}=n_{l+1}-1,i_{l+2}=n_{l+2}-1$) and $j_{k+1}=j_{k+2}=0$. It follows that if $k\geq l+1$ then $(n_{l+1}-1, j_{l+1}(S))\notin \D_{l+1}$(or $(0, j_{l+1}(S))\notin \D_{l+1}$), and if $k=l$ then $ \widehat{r}_l(n_{l+1}-1)=0$(or $\widehat{r}_l(0)=0$). Arbitrarily choose $\mathbf{w}$ and let $x=\Pi(\mathbf{w})$, then $B(x,(N^+)^{-3}R_k)\cap E \subset S$ and the conclusion holds.
\end{proof}

The following three lemmas are the key ingredients for the proof of lower dimensions. For each integer $k>0$, we write $\mathcal{S}_{k}=\mathcal{S}_{\delta}$ and $\mathcal{U}_k=\mathcal{U}_\delta$ for $\delta=(m_1 \ldots m_k)^{-1}$. For all integers $k'>k>0$, we write
\begin{eqnarray*}
&&\Gamma_{k,k'}^-(E)=\min_{S\in\mathcal{S}_{k}} \Gamma_{k,k'}(S), \qquad \Gamma_{k,k'}^+(E)=\max_{S\in\mathcal{S}_{k}} \operatorname{card}\{S'\in \mathcal{S}_{k'} : S'\subset S\}, \\
 &&  \Gamma_{k,k'}(S)=\operatorname{card}\{S'\in \mathcal{S}_{k'} : S'\subset S\}.
\end{eqnarray*}
Next lemma shows that $\Gamma_{k,k'}^-(E)$ is bounded by the number $ N_{k,k'}^-(E)$.
\begin{lem}\label{ld_cover}
For all integers $k'>k\geq 1$, let $N_{k,k'}^-(E)$ be given by\eqref{def_Nkk-}. Then
$$
\Gamma_{k,k'}^-(E) = N_{k,k'}^-(E).
$$
\end{lem}

\begin{proof}
Fix $k$ and $k'$, let $l=l(k)$ and $l'=l'(k')$ be given by~\eqref{def_l}.
For each $S(x)\in\mathcal{S}_{k}$ where $x\in S(x)\cap E$, there exists a unique $U(\mathbf{w})\in \mathcal{U}_k$, such that $S(x)=\Pi(U(\mathbf{w}))$ and $x=\Pi(\mathbf{w})$.
For each $S'(x')\in \mathcal{S}_{k'} $ such that $S'(x')\subset S(x)$, let $\mathbf{w}=w_1w_2\ldots w_n\ldots$ and $\mathbf{w'}=w_1'w_2'\ldots w_n'\ldots$ such that $\Pi(w)=x$ and $\Pi(w')=x'$, where $w_n=(i_n,j_n), w_n'=(i_n',j_n')\in \mathcal{D}_n$, and we have that
$$
\begin{array}{ll}
    i_n=i_n',&n=1,\ldots,l, \\
    j_n=j_n',&n=1,\ldots,k.
\end{array}
$$
Computing $\Gamma_{k,k'}(S(x))$ is equivalent to counting the number of $\mathbf{w'}$ such that $\Pi(\mathbf{w'})\in\mathcal{S}_{k'} $ satisfying above property. Therefore it is divided into six cases: $l < l'\leq k <k'$, $l \leq k < l' \leq k'$, $l \leq k < k' \leq l'$, $k \leq l < l' \leq k'$, $k \leq l < k' \leq l'$ and  $ k < k' \leq l < l'$. We only prove the first three cases, and the other three cases are the same by interchanging the directions.

\noindent (1) For $ l < l'\leq k <k' $.  We have that
\begin{eqnarray*}
\operatorname{card}\{i_n': (i_n',j_n)\in \mathcal{D}_n\}&=&r_n(j_n)\geq r_n^-, \qquad  \textit{  for } n=l+1,\ldots,l',
\\
\operatorname{card}\{j_n':(i_n',j_n')\in \mathcal{D}_n, \textit{ for some  }i_n'\}&=&s_n,   \qquad \qquad \qquad \ \textit{  for } n=k+1,\ldots,k'.
\end{eqnarray*}
Therefore,
\begin{eqnarray*}
  \Gamma_{k,k'}(S(x))&=&r_{l+1}(j_{l+1}) \ldots r_{l'}(j_{l'}) s_{k+1} \ldots s_{k'}   \\
    &\geq& r_{l+1}^- \ldots r_{l'}^- s_{k+1} \ldots s_{k'}\\
    &=&N_{k,k'}^-(E).
\end{eqnarray*}
Since it holds for all $x\in E$, we have that
$$
\Gamma_{k,k'}^-(E)\geq N_{k,k'}^-(E).
$$

On the other hand, we choose $\mathbf{w}=w_1 w_2\ldots w_n\ldots \in \Sigma^\infty$, where $w_n=(i_n,j_n)\in\mathcal{D}_n$ such that $r_n(j_n)=r_n^-$ for $n=1,2,3\ldots$. Let $x=\Pi(\mathbf{w})$ and $S(x)=\Pi(U(\delta,\mathbf{w}))$. Then we have that
\begin{eqnarray*}
  \Gamma_{k,k'}^-(E)&\leq& \Gamma_{k,k'}(S(x))   \\
                    &=& r_{l+1}^- \ldots r_{l'}^- s_{k+1} \ldots s_{k'}\\
                    &=&N_{k,k'}^-(E).
\end{eqnarray*}

Hence for $ l < l'\leq k <k' $, it is true that $\Gamma_{k,k'}^-(E)= N_{k,k'}^-(E).$

\noindent (2) For $ l \leq k < l' \leq k'$, we have that
\begin{eqnarray*}
\operatorname{card}\{i_n: (i_n,j_n)\in \mathcal{D}_n\}&=&r_n(j_n)\geq r_n^-, \qquad \textit{ for }  n=l+1,\ldots,k,
\\
\operatorname{card}\{(i_n,j_n):(i_n,j_n)\in \mathcal{D}_n\}&=&r_n, \qquad \qquad \qquad \textit{ for } n=k+1,\ldots,l',
\\
\operatorname{card}\{j_n:(\widetilde{i_n},j_n)\in \mathcal{D}_n \textit{ for some } \widetilde{i_n}\}&=&s_n ,\qquad \qquad \qquad \textit{ for } n=l'+1,\ldots,k'.
\end{eqnarray*}

Therefore we have that
\begin{eqnarray*}
\Gamma_{k,k'}(S(x))&=&r_{l+1}(j_{l+1}) \ldots r_{k}(j_{k}) r_{k+1} \ldots r_{l'} s_{l'+1} \ldots s_{k'} \\
  &\geq& r_{l+1}^- \ldots r_{k}^- r_{k+1} \ldots r_{l'} s_{l'+1} \ldots s_{k'} \\
  &=&N_{k,k'}^-(E).
\end{eqnarray*}
Since it holds for all $x\in E$, we have that
$$
\Gamma_{k,k'}^-(E)\geq N_{k,k'}^-(E).
$$

On the other hand, we choose $\mathbf{w}=w_1 w_2\ldots w_n\ldots \in \Sigma^\infty$, where $w_n=(i_n,j_n)\in\mathcal{D}_n$ such that $r_n(j_n)=r_n^-$ for $n=1,2,3\ldots$. Let $x=\Pi(\mathbf{w})$ and $S(x)=\Pi(U(\delta,\mathbf{w}))$. Then we have that
$$
\Gamma_{k,k'}(S(x))=r_{l+1}^- \ldots r_{k}^- r_{k+1} \ldots r_{l'} s_{l'+1} \ldots s_{k'}=N_{k,k'}^-(E).
$$
It follows that
\begin{eqnarray*}
  \Gamma_{k,k'}^-(E)&\leq & \Gamma_{k,k'}(S(x))=N_{k,k'}^-(E).
\end{eqnarray*}

Hence for $ l \leq k < l' \leq k' $, it is true that $\Gamma_{k,k'}^-(E)= N_{k,k'}^-(E).$

\noindent (3) For $ l \leq k < k' \leq l'$, we have that
\begin{eqnarray*}
\operatorname{card}\{i_n: (i_n,j_n)\in \mathcal{D}_n\}&=&r_n(j_n)\geq r_n^-, \qquad \textit{ for }  n=l+1,\ldots,k,
\\
\operatorname{card}\{(i_n,j_n):(i_n,j_n)\in \mathcal{D}_n\}&=&r_n, \qquad \qquad\qquad\textit{ for } n=k+1,\ldots,k',
\\
\operatorname{card}\{i_n:(i_n,\widetilde{j_n})\in \mathcal{D}_n \textit{ for some } \widetilde{j_n}\}&=&\widehat{s}_n ,\qquad \qquad\qquad\textit{ for } n=k'+1,\ldots,l'.
\end{eqnarray*}

Therefore we have that
\begin{eqnarray*}
\Gamma_{k,k'}(S(x))&=&r_{l+1}(j_{l+1}) \ldots r_{k}(j_{k}) r_{k+1} \ldots r_{k'} \widehat{s}_{k'+1} \ldots \widehat{s}_{l'} \\
&\geq& r_{l+1}^- \ldots r_{k}^- r_{k+1} \ldots r_{k'} \widehat{s}_{k'+1} \ldots \widehat{s}_{l'} \\
  &=&N_{k,k'}^-(E).
\end{eqnarray*}
Since it holds for all $x\in E$, we have that
$$
\Gamma_{k,k'}^-(E)\geq N_{k,k'}^-(E).
$$

On the other hand, we choose $\mathbf{w}=w_1 w_2\ldots w_n\ldots \in \Sigma^\infty$, where $w_n=(i_n,j_n)\in\mathcal{D}_n$ such that $r_n(j_n)=r_n^-$ for $n=1,2,3\ldots$. Let $x=\Pi(\mathbf{w})$ and $S(x)=\Pi(U(\delta,\mathbf{w}))$. Then we have that
\begin{eqnarray*}
  \Gamma_{k,k'}^-(E)&\leq & \Gamma_{k,k'}(S(x))\\
                      &=&r_{l+1}^- \ldots r_{k}^- r_{k+1} \ldots r_{k'} \widehat{s}_{k'+1} \ldots \widehat{s}_{l'}\\
                      &=& N_{k,k'}^-(E).
\end{eqnarray*}

Hence for $ l \leq k < k' \leq l' $, it is true that $\Gamma_{k,k'}^-(E)= N_{k,k'}^-(E).$ Therefore the conclusion holds.
\end{proof}

\begin{lem}\label{lemL} Given $\beta>0$.
There exists a constant $C$ such that $N_{k,k'}^-(E) > C (\frac{R_k}{R_{k'}})^\beta$ for all $1 \leq k \leq k'$, if and only if there exists a constant $C'$ such that $\inf_{x \in E} N_r(B(x,R) \cap E) > C' (\frac{R}{r})^\beta$ for all $0<r<R<\frac{1}{N^+}$.
\end{lem}

\begin{proof}
For all reals $r, R$ satisfying $0<r<R<\frac{1}{N^+}$, there exist integers $k,k'$ such that
$$
R_{k'} \leq r < R_{k'-1},\qquad R_{k} \leq R < R_{k-1}.
$$
Immediately, we have that
\begin{equation}\label{Rr}
(N^+)^{-\beta}  \left(\frac{R}{r}\right)^\beta\leq \left(\frac{R_k}{R_{k'}}\right)^\beta \leq (N^+)^\beta \left(\frac{R}{r}\right)^\beta.
\end{equation}

First, assume that $N_{k,k'}^-(E) > C (\frac{R_k}{R_{k'}})^\beta$ for every $1 \leq k \leq k'$.
Arbitrarily choose $x\in E$. The ball $B(x,2R)$ contains at least one approximate square in $\mathcal{S}_k$, and any set with diameter no more than $r$ intersects at most $(N^+ +1)^3$ approximate squares in $\mathcal{S}_{k'}$. Hence for all $0<r<R<\frac{1}{N^+}$, by Lemma~\ref{ld_cover} and \eqref{Rr}, we have that
\begin{eqnarray*}
N_r(B(x,2R) \cap E) &\geq& (N^+ +1)^{-3} N_{k,k'}^-(E) \\
&>& (N^+ +1)^{-3} C \left(\frac{R_k}{R_{k'}}\right)^\beta \\
&\geq& (N^+ +1)^{-3} C (C_1N^+)^{-\beta} \left(\frac{2R}{r}\right)^\beta.
\end{eqnarray*}
By taking $C'=(N^+ +1)^{-3} C (2N^+)^{-\beta}$, we have that
$$
\inf_{x\in E} N_r(B(x,R) \cap E)> C'\left(\frac{R}{r}\right)^\beta
$$
for all $0<r<R<\frac{1}{N^+}$.

Next, assume that $\inf_{x \in E} N_r(B(x,R) \cap E) > C' (\frac{R}{r})^\beta$ holds for any $0<r<R<\frac{1}{N^+}$.
Therefore, by Lemma~\ref{ld_cover} and \eqref{Rr}, for all $1 \leq k \leq k'$ and $S\in \mathcal{S}_k$, if $\frac{R_k}{R_{k'}}>2(N^+)^3$, we have that
\begin{eqnarray*}
\Gamma_{k,k'}(S) &\geq& N_{2R_{k'}}(B(x,(N^+)^{-3}R_k) \cap E)\\
&>& C' 2^{-\beta} (N^+)^{-3\beta} \left(\frac{R_k}{R_{k'}}\right)^\beta.
\end{eqnarray*}
and if $\frac{R_k}{R_{k'}}\leq 2(N^+)^3$, we have that
\begin{eqnarray*}
\Gamma_{k,k'}(S) &\geq& 1 > \frac{2^{-\beta} (N^+)^{-3\beta}}{2} \left(\frac{R_k}{R_{k'}}\right)^\beta.
\end{eqnarray*}
By taking $C=\min\{C' 2^{-\beta} (N^+)^{-3\beta}, \frac{2^{-\beta} (N^+)^{-3\beta}}{2}\}$, we have that
$$
N_{k,k'}^-(E) > C \left(\frac{R_k}{R_{k'}}\right)^\beta
$$
for all $1 \leq k \leq k'$. Then the conclusion holds.
\end{proof}

We write that
$$
\Psi_{k,k'}(\xi)= N_{k,k'}^-(E) (m_{k+1} \ldots m_{k'})^{-\xi}.
$$
Clearly, the function $\Psi_{k,k'}(\xi)$ is decreasing in $\xi$. For all $k<k'$, we write $\xi_{k,k'}$ for the unique solution $\Psi_{k,k'}(\xi)=1$, and it is clear that
\begin{equation}\label{xik}
\xi_{k,k'}=\frac{\log N_{k,k'}^-(E)}{\log m_{k+1} \ldots m_{k'}}.
\end{equation}

For all integers $k''>k'>k>1$,  by Lemma~\ref{ld_cover}, we have that
$$
N_{k,k''}^-(E)\geq N_{k,k'}^-(E) N_{k',k''}^-(E).
$$
Immediately, it follows that
\begin{equation}\label{psikk}
\Psi_{k,k''}(\xi)\geq \Psi_{k,k'}(\xi) \Psi_{k',k''}(\xi).
\end{equation}

\begin{lem}\label{ld_cvgt}
The sequence $\{\inf_k \xi_{k,k+m}\}_{m=1}^\infty$ is convergent.
\end{lem}

\begin{proof}
For each integer $m>0$, we write $\zeta_m= \inf_k \xi_{k,k+m}.
$
Since $\zeta_m\leq \xi_{k,k+m}$,  it is clear that for all integers $i>0$ and $k>0$,
$$
\Psi_{k+im,k+(i+1)m}(\zeta_m) \geq 1.
$$
Fix an integer $m >0$. For each $\xi < \zeta_m$,   for all integers $k>0$, $p>0$ and $n>0$ such that $0\leq n\leq m-1$, by~\eqref{psikk}, we obtain that
\begin{eqnarray*}%\label{}
  \Psi_{k,k+pm+n}(\xi) &\geq& \left(\prod_{i=0}^{p-1} \Psi_{k+im,k+(i+1)m}(\xi)\right)\cdot \Psi_{k+pm,k+pm+n}(\xi)\\
  &=& \left(\prod_{i=0}^{p-1} \Psi_{k+im,k+(i+1)m}(\zeta_m)\big(m_{k+im+1}\ldots m_{k+(i+1)m}\big)^{\zeta_m-\xi}\right)\\
  &&  \qquad\qquad\quad \cdot\Psi_{k+pm,k+pm+n}(\zeta_n)\Big(m_{k+pm+1}\ldots m_{k+pm+n}\Big)^{\zeta_n-\xi}\\
  &\geq& 2^{p(\zeta_m-\xi)} \min\left\{2^{n(\zeta_n-\xi)}, (N^+)^{n(\zeta_n-\xi)}\right\}.
\end{eqnarray*}
Since  $\xi < \zeta_m$, there exists an integer $K_0$ such that for all $p \geq K_0$,
$$
\Psi_{k,k+pm+n}(\xi) \geq 1.
$$
Hence, for all integers $p \geq K_0$, $k \geq 0$ and $n$ such that $0\leq n\leq m-1$, we have that $\xi_{k,k+pm+n} \geq \xi, $
and this implies that $\zeta_{pm+n}\geq \xi.$
Therefore, for all integers $n$ such that $0\leq n\leq m-1$, we have that
$
\liminf_{p \to \infty}\zeta_{pm+n} \geq \xi.
$
Since it holds for all $\xi < \zeta_m$, we obtain that
$$
\liminf_{m \to \infty} \zeta_m \geq \limsup_{m \to \infty} \zeta_m.
$$
Therefore $\{\zeta_m\}$ is convergent, and the conclusion holds.
\end{proof}

To prove the Assouad dimension, we need the following three lemmas. Since Assouad dimension is the dual of lower dimension, the proofs of these lemmas are similar to the lemmas used for lower dimensions, and we skip these proofs.

\begin{lem}\label{lem_cover}
For all integers $k'>k\geq 1$, we have that
$$
\Gamma_{k,k'}^+(E) = N_{k,k'}^+(E).
$$
\end{lem}

\begin{lem}\label{lemA} Given $\beta>0$.
There exists a constant $C$ such that $N_{k,k'}^+(E) < C (\frac{R_k}{R_{k'}})^\beta$ for all $1 \leq k \leq k'$, if and only if there exists a constant $C'$ such that $\sup_{x \in E} N_r(B(x,R) \cap E) < C' (\frac{R}{r})^\beta$ for all $0<r<R<\frac{1}{N^+}$.
\end{lem}
We write that
$$
\Delta_{k,k'}(\beta)= N_{k,k'}^+(E) (m_{k+1} \ldots m_{k'})^{-\beta}
$$
and we write $\beta_{k,k'}$ for the unique solution $\Delta_{k,k'}(\beta)=1$
\begin{lem}\label{lem_cvgt}

The sequence $\{\sup_k \beta_{k,k+m}\}_{m=1}^\infty$ is convergent.
\end{lem}

Now, we are ready to prove the lower and Assouad dimension of self-affine dimensions

\begin{proof}[Proof of Theorem~\ref{thmL}]
By \eqref{xik} and Lemma~\ref{ld_cvgt}, we write that
\begin{equation}\label{xi*}
\xi_*=\lim_{m \to \infty} \inf_{k}\xi_{k,k+m}=\lim_{m \to \infty} \inf_{k}\frac{\log N_{k,k+m}^-(E)}{\log m_{k+1} \ldots m_{k+m}}.
\end{equation}

To prove that t $\xi_*$ is  the upper bound for the lower dimension of $E$,  we choose a subsequence $\{(k_n,k_n^\prime)\}_{n=1}^\infty$ such that $\lim_{n\to\infty}k_n'-k_n=\infty$ and  $\lim_{n \to \infty} \xi_{k_n,k_n^\prime} = \xi_*.$
For all $\xi > \xi_*$, there exists an integer $K'>0$ such that, for all $n > K'$, $$
\xi > \xi_{k_n,k_n^\prime}.
$$
Combining with ~\eqref{xik},  we obtain that
$$
N_{k_n,k_n^\prime}^-(E)=\bigg(\frac{R_{k_n}}{R_{k_n^\prime}} \bigg)^{\xi_{k_n,k_n^\prime}}\leq \bigg(\frac{R_{k_n}}{R_{k_n^\prime}}\bigg)^\xi.
$$
For all $\varepsilon>0$, by the definition of lower dimension, there exists $C_\varepsilon$ such that
$$
\inf_{x\in E}N_{r}(B(x,R)\cap E) \geq C_\varepsilon \bigg(\frac{R}{r}\bigg)^{\lwd E - \varepsilon}.
$$
By Lemma~\ref{lemL}, this is equivalent to
$$
N_{k_n,k_n^\prime}^-(E) \geq C_\varepsilon \bigg(\frac{R_{k_n}}{R_{k_n^\prime}}\bigg)^{\lwd E-\varepsilon},
$$
for all $n>0$. Immediately, we obtain that
$$
\bigg(\frac{R_{k_n}}{R_{k_n^\prime}}\bigg)^\xi \geq C_\varepsilon \bigg(\frac{R_{k_n}}{R_{k_n^\prime}}\bigg)^{\lwd E-\varepsilon},
$$
and it implies that
$$
\xi \geq \lwd E-\varepsilon+\frac{\log C_\varepsilon}{\log R_{k_n}-\log R_{k_n^\prime}}.
$$
by taking  $n$ tend to $ \infty$, we have that $\xi \geq  \lwd E-\varepsilon$. Since $\varepsilon $ is arbitrarily chosen,   we obtain that
$$
\lwd E \leq \xi,
$$
for all $\xi > \xi_*$. Thus the inequality $\lwd E \leq \xi_*$ holds.

Next, we prove that $\xi_*$ is the lower bound. The  conclusion holds for $\xi_*=0$, and we only consider that $\xi_*>0$.

Arbitrarily choose $0< \xi < \xi_* $, by~\eqref{xi*}, there exists an integer $K''>0$ such that, for all $m > K''$, we have that $\xi < \xi_{k,k+m}$,  for all integers $k>0$ . Combining with ~\eqref{xik}, it follows that for all $k'-k > K''$,
$$
N_{k,k'}^-(E) = (m_{k+1}\ldots m_{k'})^{\xi_{k,k'}} = \bigg(\frac{R_k}{R_{k'}}\bigg)^{\xi_{k,k'}} >  \bigg(\frac{R_k}{R_{k'}}\bigg)^\xi.
$$
For all $k'-k \leq K''$, since $\xi>0$, we obtain that
\begin{eqnarray*}\label{}
N_{k,k'}^-(E) &\geq &  \bigg(\frac{R_k}{R_{k'}}\bigg)^{\xi-\xi}
\geq \bigg(\frac{R_k}{R_{k'}}\bigg)^\xi (N^+)^{-K'' \xi}.
\end{eqnarray*}
Let $C_\xi=(N^+)^{-K'' \xi}$, we have
$$
N_{k,k'}^-(E) \geq C_\xi \bigg(\frac{R_k}{R_{k'}}\bigg)^\xi
$$
for all $\xi > \xi_* $. By Lemma~\ref{lemL}, $\lwd E \geq \xi_*$. Hence the lower dimension formula holds.

The proof for Assouad dimension is similar to the lower dimensions, where Lemma~\ref{ld_cover}, Lemma~\ref{lemL} and Lemma~\ref{ld_cvgt} are replace by Lemma~\ref{lem_cover}, Lemma~\ref{lemA} and Lemma~\ref{lem_cvgt} and we omit it.
\end{proof}

\iffalse
\section*{Acknowledgments}
The authors wish to thank Prof. Kenneth
Falconer  and Prof. Wenxia Li for their helpful comments. This research  is partially supported by Shanghai Key Laboratory of PMMP (18dz2271000).
\fi

\end{document}